\documentclass[12pt, a4paper]{amsart}
\usepackage{amsmath}
\usepackage{amssymb}
\usepackage{amsthm}
\usepackage{breqn}
\setkeys{breqn}{compact}
\usepackage{t1enc}
\usepackage{tznotes2}
\usepackage[shortlabels,inline]{enumitem}
\newcommand{\R}{\mathbb{R}}%
\newcommand{\Z}{\mathbb{Z}}%
\newcommand{\N}{\mathbb{N}}%
\newcommand{\II}{\mathcal{I}}%
\newcommand{\JJ}{\mathcal{J}}%
\newcommand{\KK}{\mathcal{K}}%
\newcommand{\leb}{\mathcal{L}}%
\newcommand{\setsep}{\colon\,}
\newcommand{\eqdef}{\mathrel{\mathop:}=}%
\newcommand{\Lip}{\operatorname{Lip}}%
\newcommand{\lip}{\operatorname{lip}}%
\newcommand{\dist}{\operatorname{dist}}%
\newcommand{\osc}{\operatorname{osc}}%
\newcommand{\spt}{\operatorname{supp}}%
\newcommand{\supp}{\operatorname{supp}}%
\newcommand{\eps}{\varepsilon}%
\newcommand*{\Char}{\text{\usefont{U}{bbold}{m}{n}1}}%
\let\leq\leqslant%
\let\geq\geqslant%
\let\tilde\widetilde%
\let\hat\widehat%
\newcommand{\field}[1]{\mathbb{#1}}

\usepackage{mathtools}
\DeclarePairedDelimiter\abs{\lvert}{\rvert}%
\DeclarePairedDelimiter\norm{\lVert}{\rVert}%
\makeatletter
\let\oldabs\abs
\def\abs{\@ifstar{\oldabs}{\oldabs*}}
\let\oldnorm\norm
\def\norm{\@ifstar{\oldnorm}{\oldnorm*}}
\makeatother
\newcommand{\dt}{\ensuremath{\,\mathrm{d}t}}%
\catcode`\=13%
\def{\ensuremath{\bowtie}}%

\usepackage[capitalize,nameinlink]{cleveref}
\crefname{empty}{}{}
\crefname{equation}{}{}
\Crefname{figure}{Figure}{Figures}
\crefname{page}{page}{pages}
\Crefname{enumi}{}{}
\Crefname{subsection}{Subsection}{Subsections}
\def\theoremname{Theorem}%
\def\propositionname{Proposition}%
\def\lemmaname{Lemma}%
\def\corollaryname{Corollary}%
\def\definitionname{Definition}%
\def\conventionname{Convention}%
\def\axiomname{Axiom}%
\def\remarkname{Remark}%
\def\examplename{Example}%
\def\questionname{Question}%
\def\constructionname{Construction}%
\def\assumptionname{Assumption}%
\def\conjecturename{Conjecture}%
\newtheorem{thm}{\theoremname}[section]
\theoremstyle{plain}
\newtheorem{theorem}[thm]{\theoremname}

\newtheorem{lemma}[thm]{\lemmaname}
\newtheorem{lem}[thm]{Lemma}

\theoremstyle{definition}
\newtheorem{definition}[thm]{\definitionname}

\newtheorem{notation}[thm]{Notation}

\numberwithin{equation}{section}%
\usepackage{esvect}
\newcommand{\future}[1]{\vv{#1}}
\newcommand{\past}[1]{#1}

\begin{document}

\title[On sets where $\lip f$ is infinite]{On sets where $\lip f$ is infinite\\for monotone continuous functions}%

\author{Martin Rmoutil}
\address{Department of Mathematics Education, Faculty of Mathematics and Physics, Charles University, Sokolovská 83, 186 75 Praha 8, Czech Republic}
\email{rmoutil@karlin.mff.cuni.cz}

\author{Thomas Zürcher}
\address{Instytut Matematyki, Uniwersytet Śląski, Bankowa 14, 40-007 Katowice, Poland}
\email{thomas.zurcher@us.edu.pl}



\date{\today}%
\dedicatory{As a teacher and coauthor, our beloved colleague Jan Malý helped to shape our mathematical and world view.
He is sorely missed.
With admiration, we dedicate this paper to him.}

\begin{abstract}
  For a function $f\colon \field{R}\to \field{R}$, we let 
  \begin{equation*}
    \lip f(x)=\liminf_{r\to 0_+}\sup_{y\in [x-r,x+r]} \frac{\abs{f(y)-f(x)}}{r}.
  \end{equation*}
  Given any $F_{\sigma\delta}$ set $A\subseteq \R$ of Lebesgue measure zero, we explain how to obtain a nondecreasing absolutely continuous function $g\colon\R\to [0,1]$ such that  $g'(x)=\infty$ for every $x\in A$, and $\lip g(x)<\infty$ for every $x\notin A$.
\end{abstract}
\maketitle
\section{Introduction}
In~\cite{Rademacher}, H.~Rademacher proved that Lipschitz functions between Euclidean spaces are differentiable almost everywhere, see Satz~I.
As we look at functions on the real line, we mention that the one-dimensional case (actually for the larger class of functions of bounded variation) is due to Lebesgue, see page~128 in~\cite{Lebesgue}, republished in~\cite{Lebesgue2009-ms}.
A strengthening of Rademacher's result is by W.~Stepanov in~\cite{Stepanoff}, but before detailing it, let us introduce some notation.
Although our setting is the real line, the following definition is of a metric flavour, and hence we give it for metric spaces.
\begin{definition}\label{D:lipLip}
  Let $(X,d_X)$ and $(Y,d_Y)$ be metric spaces and $f\colon X\to Y$ be a mapping.
  Then we define
  \begin{align*}
      \lip f(x)\eqdef\liminf_{r\to 0_+}\sup_{y\in B(x,r)} \frac{d_Y(f(y),f(x))}{r},\\
      \Lip f(x)\eqdef\limsup_{r\to 0_+}\sup_{y\in B(x,r)} \frac{d_Y(f(y),f(x))}{r}.      
  \end{align*}
  In case we have $X,Y\subseteq\field{R}$, we sometimes replace $B(x,r)$ in the above formulae by $(x-r,x]$ and $[x,x+r)$.
  We indicate this by using $\Lip(x_-)$ and $\Lip(x_+)$, respectively and the same for~$\lip$.
\end{definition}
Stepanov proved that functions $f\colon \mathbb{R}^m\to \mathbb{R}^n$  are differentiable at almost every point in the set where $\Lip f$ is finite.
We also would like to mention the slick proof of this statement by J.~Malý in~\cite{Maly}.
Z.~Balogh and M.~Csörnyei showed in \cite{BC06} that there are functions~$f$ that fail to be differentiable at almost every point in the set where $\lip f$ is finite, see Theorems 1.3~and~1.4 in their paper highlighting two different issues.
However, they also showed in Theorem~1.2 that Stepanov-type theorems still hold if the integrability of $\lip$ and the size of the set where $\lip$ is infinite satisfy certain carefully balanced restrictions.

While these results shed light on the size of the set where $\lip$ is infinite and its connection to differentiability, in the current paper, we are interested to know more about the structure of this set.
This is in the tradition of studying the structure of the set where the derivative of a function is infinite.
Note that $\lip f(x)=\infty$ whenever $f'(x)=\infty$.
A testament to the inquiry of the sets where the derivative is infinite are for example the two papers (written in German) by V.~Jarník, \cite{Jarnik}, and by Z.~Zahorski, \cite{Zahorski}.
Jarník proved in~Satz~3 that given a $G_\delta$ set $G\subset \field{R}$ of measure zero, there is a nondecreasing continuous function $f\colon \field{R}\to \field{R}$ such that $f'(x)=\infty$ for $x\in G$ and all Dini derivatives are finite for $x\notin G$.
Zahorski found such a function that has an infinite derivative at each point in~$G$ and a finite derivative at each point in the complement of~$G$.

The analysis for $\Lip$ is much easier to do than the one for~$\lip$.
For example, Theorem~3.35 and Lemma~2.4 in~\cite{BHRZ} show that, given a set $A\subseteq \mathbb{R}$, there is a continuous function $f\colon \mathbb{R}\to \mathbb{R}$ satisfying $A=\{x\in\R\setsep \Lip f(x)=\infty\}$ if and only if $A$ is a $G_{\delta}$ set.
The just mentioned Lemma~2.4 also states that the set where $\lip$ is infinite for a continuous function is an $F_{\sigma\delta}$ set.
We conjecture that whenever $A$ is an $F_{\sigma\delta}$ set in the real line, there exists a continuous function $f\colon \mathbb{R}\to \mathbb{R}$ such that $A=\{x\in \mathbb{R}\setsep \lip f(x)=\infty\}$.
In~\cite{BHRZ}, such functions are given in case $A$ is an $F_{\sigma}$ set.
In this paper, we replace the assumption that $A$ be an $F_{\sigma}$ set by $A$ being an $F_{\sigma\delta}$ set of measure zero.
Moreover, the fact that the set has vanishing measure enables us to find appropriate functions that are not only continuous but actually absolutely continuous and nondecreasing.

To close the introduction, we look at some further papers connected to our research.
One part of B.~Hanson's Theorem~1.3 in~\cite{Bruce2} tells us that if $E\subset \field{R}$ is a $G_\delta$ set with measure zero, then there exists a continuous, monotonic function $f\colon \field{R}\to \field{R}$ such that $E=\{x\in\R\setsep \Lip f(x)=\infty\}$ and $\{x\in\R\setsep \lip f(x)=\infty\}=\emptyset$.
Moreover $f$ may be constructed so that $\lip f(x)=0$ for all $x\in E$.
Theorem~1.2 deals with the case where $f$ is not monotonic.


A similar topic as the one in our paper is the study of sets~$E$ such that there is a continuous function $f$ satisfying $\lip f(x)=\chi_E(x)$ (or $\Lip f(x)=\chi_E$).
We refer the interested reader to the papers~\cite{BHMV20a,BHMV20b,BHMV21}, where characterizations of such sets are found.

\section{The statement of the main result and an explanation of the strategy of its proof}

\begin{theorem}[main result]\label{T:main}
    Assume that $A\subset \mathbb{R}$ is an $F_{\sigma \delta}$ set of measure zero, then there exists a nondecreasing absolutely continuous function $g\colon \mathbb{R}\to \mathbb{R}$ such that $g'(x)=\lip g(x)=\infty$ for every $x\in A$ and $\lip g(x)<\infty$ for every $x\notin A$.
\end{theorem}

We already mentioned the results by Jarník and Zahorski in the introduction.
Our proof actually makes use of the function that Jarník found.
Moreover, our overall proof strategy shows similarities to the one employed by Jarník.
Neither Jarník nor Zahorski state that the constructed function is absolutely continuous.



Let us give a brief description of Jarník's construction with an interwoven argument why the constructed function is absolutely continuous.

He starts with an arbitrary $G_\delta$ set $G$ written as countable intersection of open sets~$O_n$.
In the next step, he replaces the open sets by better suited open sets~$U_n$.
Following this, he defines $f_k(x)=\leb^1((-\infty, x)\cap U_k)$.
Finally he adds up all the functions $f_k$ to obtain a function~$f$ having the claimed properties.
That $f'(x)=\infty$ for $x\in G$ follows since $f_k'(x)=1$ for all such $x$.
Note that $f_k'(x)=0$ for~$x$ outside the closure of~$U_k$.
The sets~$U_k$ are chosen so small that $\sum_{k=1}^\infty \norm{f_k}_{1}<\infty$.
This guarantees that $f$ is absolutely continuous, see~\cref{sum AC}.

That the Dini derivatives are finite outside~$G$ needs the cleverly chosen sets~$U_k$ and some careful computations and estimates.


Behind the choice of~$f_k$ lies the fact that given a set $A$, setting $f(x)=\int_{0}^{x}\chi_{A}$ guarantees that $f'(x)=1$, whenever $x$ is a density point of~$A$.

In Jarník's case, the main reason for the modification of the sets~$O_n$ to the sets~$U_n$ is to make sure that the points outside the intersection of all these open sets are such that the function has finite derivative there; in our case we also need to focus on the points in the intersection.
In some sense, we need to transform these points into density points.

Having reviewed Jarník's proof strategy, it is now time to talk about the ideas behind the results in our paper.
\bigskip

The main part of our main result is covered by \cref{null}, which is almost our main result \cref{T:main}, but contains the additional assumption that the $F_{\sigma\delta}$ set $A$ be meagre. A brief outline and an explanation of the proof follow.

The desired function $g$ will be constructed in the form $\sum_{k=1}^\infty g_k$, where each $g_k$ will be
carefully crafted to have the desired properties.
We record the creation of these functions in \cref{L:functiong} and its proof
allowing for particular choices of the parameters of the lemma, especially the sets $E,F, H$.

First we express the set $A$ in a more convenient form, namely we find closed sets $F_k$ ($k\in\N$) such that $A$ is exactly the set of points belonging to infinitely many $F_k$'s. Moreover, the sets $F_k$ are chosen so that for any indices $k,l\in\N$ with $k<l$, if $F_k\cap F_l\neq \emptyset$, then $F_l\subseteq F_k$. Our method of proof relies heavily on these properties. It is useful to note that such an arrangement of closed sets is essentially a level-disjoint Suslin scheme. The disjointness and nestedness are essential for our method, and they can be achieved thanks to the assumption that $A$ be meagre; without it, we can run into problems:
For example, let $A=(0,1)$ and $A=\bigcap L_n$ with $L_n$ of the type $F_\sigma$ for every $n$, and $L_n\subseteq L_m$ whenever $n>m$. Then for some $n_0$ and any $n\geqslant n_0$ we have $0,1\notin L_n$, and by intersecting with $[0,1]$, we may assume also that $L_n\subseteq [0,1]$, so $L_n=(0,1)$. But then, by a classical result of W.~Sierpi\'{n}ski \cite{Sierpinski}, $L_n$ cannot be expressed as the union of countably many pairwise disjoint closed sets.




To resume our thoughts about the Suslin scheme, we note that in particular, the sets $F_k$ are naturally arranged (by inclusion) in a tree. Any vertex of this tree (i.e.\ any of the sets $F_k$) can be seen as the root of a subtree, which we (for the purposes of this explanatory remark) call a family of (all descendants of) the set $F_k$. Clearly, if we pick two indices $k,l$ such that $F_k\cap F_l=\emptyset$, then the corresponding subtrees (families) are also disjoint, and any member of one subtree is disjoint from any set of the other. The two families, i.e.\ the one started by $F_k$ and the one started by $F_l$, can thus be seen as ``unrelated''. For the purposes of this remark, we shall use the words ``descendant'' and ``ancestor'' in the obvious sense of the tree order, while the words ``previous'' or ``past'', and ``later'' or ``future'' will refer simply to the order of indices; that is, if $j<k$, then $F_j$ is previous to $F_k$ and $F_k$ is future to $F_j$.

Next we use the fact that $A$ is Lebesgue null to find pairwise disjoint measurable sets $M_k$ contained in the complement~$A^c$ of~$A$, each with positive measure in every nonempty open interval in $\R$. These are used in \cref{L:SetH} to obtain the compact sets $H_k$ satisfying $F_k\subseteq H_k\subseteq F_k\cup M_k$; the set $H_k$ is where the function $g_k$ is later allowed to grow (see \cref*{L:functiong}~\cref{LI:1}). Requirement~\cref{LI:2} in~\cref*{L:functiong} on the function makes it clear that the set $F_k$ itself is not sufficient for this purpose and must be enlarged.

Now we make the important step to define the closed sets $\future{H}_k=\bigcup H_j$ where the union is over all $j$ with $F_j\subseteq F_k$, i.e.\ over all the family of $F_k$. It can be seen from the definition that the ordering by inclusion of the sets $\future{H}_k$ is the same as that of the sets $F_k$ in the sense that if $F_k\subseteq F_l$, then also $\future{H}_k\subseteq \future{H}_l$.
However, $\future{H}_k$ and $\future{H}_j$ need not be disjoint, even if $F_k$ and $F_j$ are, a fact that is a source of some complications. The set $\future{H}_k$ contains $F_k$ (as well as all of its descendants, of course) as its ``core'' and it also covers the whole space in which any of the corresponding descendant functions (i.e.\ all $g_j$ for $j$ such that $F_j\subseteq F_k$) will be allowed to grow. The notation tries to convey that $\future{H}_k$ takes responsibility for the whole future of the family.

We also define the closed sets $E_k=\bigcup \future{H}_j$ where the union is over all $j<k$ (so the union is already clear to be over finitely many sets) with $F_j\cap F_k=\emptyset$, that is, over ``\emph{previous} families unrelated to that of $F_k$''. The union is finite, so only finitely many families are involved; on the other hand, each of them is involved as a whole because we use the sets $\future{H}_j$ (and not just $H_j$) in the definition of $E_k$.  

At this point we are finally ready to apply, for each $k\in\N$, \cref{L:functiong} with $E,F,H$ replaced by $E_k, F_k, H_k$, and obtain the functions $g_k$. We may adopt the ``family terminology'' also for the functions, e.g.\ $g_j$ is a descendant of $g_k$ if $F_j\subseteq F_k$. 

The role of the sets $E_k$ is key and is largely revealed by \cref*{L:functiong}~\cref{LI:4}: loosely speaking, $g_k$ is forced to ``behave nicely'' in the vicinity of $E_k$. But $E_k$ is the set where the previous unrelated functions (i.e.\ with smaller indices from different families), as well as all their descendants, are allowed to grow. This means that $g_k$ is chosen so carefully that it does not ``provide unsolicited support to the achievements of \emph{previous} unrelated functions''; simply put, the growths of all the functions $g_k$ do not add up too much where they should not, and this allows us later to prove that for $g\eqdef \sum_{k=1}^\infty g_k$ we have $\lip g(x)<\infty$ whenever $x\notin A$. 
Another way to put this is as follows: For any two \emph{unrelated} functions, say $g_j$ and $g_k$ with $j<k$, we have $H_j\subseteq E_k$ (even $\future{H}_j\subseteq E_k$) and so $g_k$ must ``behave nicely'' close to where $g_j$ grows. That is, the later of the two unrelated functions is taking responsibility for the control we need. Of course, this is just a very rough general idea. The remainder of the proof is to show precisely that $g$ indeed enjoys the desired properties.

It is easy to show that $g'(x)=\infty$ at each point $x\in A$, as $x$ belongs to $F_k$ for infinitely many indices $k$, and for such $k$ we have $g_k'(x)=1$ by \cref*{L:functiong}~\cref{LI:3}; a notable condition in \cref{LI:3} is that $x\notin E_k$, and it must therefore be shown that this is the case whenever $x\in A\cap F_k$. This is not as trivial as it might seem at a first glance because -- as mentioned above -- the sets $\future{H}_k$ and $\future{H}_j$ are not necessarily disjoint even if $F_k$ and $F_j$ are. But it follows from the construction that these intersections are contained in the sets $M_k$, which are all disjoint from~$A$; since $x\in A$, we indeed get that $x\notin E_k$.

Assume, now, that $x\notin A$. Showing that this implies $\lip g(x)<\infty$ is more involved, and it seems to require the careful preparation above. Since $x\in F_k$ for finitely many $k$, one can easily show the same also for the sets $\future{H}_k$, so let $l$ be the largest index with $x\in \future{H}_l$. Since $\Lip g_k(x)<\infty$ for every $k$, we do not have to care about finitely many summands $g_k$. Hence, we only look at indices $k>l$, in particular, indices $k$ such that $x\notin \future{H}_k$. So let $k>l$. If $x\in E_k$, then any interval containing $x$ meets~$E_k$, and the function $g_k$ satisfies the strong estimates \cref*{L:functiong}~\cref{LI:4}; this takes care of all $g_k$ for $k$ such that $x\in E_k$. 

The core of our argument deals with the set of indices $\JJ\eqdef\{k>l\setsep x\notin E_k\}.$ We set $h=\sum_{j\in\JJ} g_j$ and want to prove that $\lip h(x)<\infty$. So we wish to find a suitable decreasing sequence of radii $(r_p)_{p=1}^\infty$, one that witnesses that the lower limit in the definition of $\lip h(x)$ is finite.
We define the radii as follows: let $j_1\in\JJ$ be minimal such that $\future{H}_{j_1}$ meets the interval $(x-1,x+1)$, and $r_1=\dist(x,\future{H}_{j_1})$. Next, let $j_2\in\JJ$ be minimal such that $\future{H}_{j_2}$ meets $(x-r_1,x+r_1)$, and set $r_2=\dist(x,\future{H}_{j_2})$. We continue this process; if it stops after finitely many steps, it means that $h$ is constant on an open neighbourhood of~$x$. Similarly, if $R\eqdef\lim_{p\to\infty} r_p>0$ we easily reach the same conclusion: $h$ is constant on $(x-R,x+R)$. 

The main case is when $\lim_{p\to\infty} r_p=0$. To treat it, we fix an arbitrary $j\in\JJ$ and aim to estimate the oscillation of $g_j$ on every interval $I_p\eqdef(x-r_p,x+r_p)$ for $p\in\N$. So we fix a $p\in\N$. If $I_p\cap H_j=\emptyset$, then $g_j$ is constant on $I_p$, so assume $I_p\cap H_j\neq\emptyset$. Then we obtain that $j>j_p$, and it also follows that $g_j$ belongs to an unrelated family, i.e.\ $F_j\cap F_{j_p}=\emptyset$.  Thus, by the definition of $E_j$, we have that $H_{j_p}$ (even $\future{H}_{j_p}$) is contained in $E_j$. (Indeed, $E_j$ ``looks at past unrelated families''; but $j_p$ is the ``past'' as $j>j_p$.) But the content of \cref*{L:functiong}~\cref{LI:4} is to provide an oscillation estimate for $g_j$ on any interval that meets $E_j$, in particular on the interval $\overline{I_p}$ that clearly meets $E_j$, even $\future{H}_{j_p}$. Summing over $j\in\JJ$ we find that also $h$ ``behaves nicely'' on $I_p$, for any $p$. Thus we obtain that $\lip h(x)<\infty$. 

The preceding paragraph describes the central argument of the proof. It reveals the reason for using the sets $\future{H}_k$, especially in the definition of $E_k$, instead of just $H_k$: without this trick we would not be able to prove that $j>j_p$. Therefore we would not necessarily have that the appropriate set (in this case that would be $H_{j_p}$) is contained in $E_j$, and in turn, $g_j$ would not be guaranteed to ``behave nicely'' in $I_p$.

\section{Preliminaries}
Here we list some of the notation and conventions that we use, besides the notions of $\lip$ and $\Lip$ introduced in \cref{D:lipLip}. Throughout the paper, we work with the real line $\R$, functions from $\R$ to $\R$, the Lebesgue measure on $\R$ etc. We follow the convention $0\notin\N$, and we use the term ``countable'' for ``at most countable''. We use the standard notation $(a,b)$ for open intervals, and $[a,b]$ for closed intervals in~$\R$.
When we talk about intervals, then we tacitly assume that they are nontrivial, i.e.\ $a<b$.
Given $x\in\R$ and $r>0$, by $B(x,r)$ we mean the usual open ball (with respect to the usual metric on $\R$), i.e.\ the open interval $(x-r,x+r)$; of course, any bounded open interval is an open ball, a fact we will occasionally use without further explanation. For $x\in \R$ and $A\subseteq \R$, we write $\dist(x,A)\eqdef\inf\{\abs{x-y}\setsep y\in A\}$, the distance of $x$ from~$A$. Given any set $A\subseteq \R$, we denote its complement~$\R\setminus A$ by~$A^c$, its closure by $\overline{A}$ and its boundary~$\overline{A}\cap\overline{A^c}$ by $\partial{A}$; since we only use the notion of boundary for intervals, we could equivalently say that, for an interval $I$, $\partial{I}$ is the set of its (at most two) endpoints. For a Lebesgue measurable (or just ``measurable''), $A\subseteq\R$ we use the symbol $\abs{A}$ to denote its Lebesgue measure, or $\abs{A}=\leb^1(A)$. A set in~$\R$ is \emph{nowhere dense} if its closure has empty interior; a set $A\subseteq\R$ is said to be \emph{meagre} if it can be written as the union of countably many nowhere dense sets.

For any function $f\colon \R\to \R$ and a set $U\subseteq\R$, we denote by $\osc(f,U)$ the \emph{oscillation} of $f$ over~$U$, i.e.\ $\osc(f,U)\eqdef\sup\{\abs{f(x)-f(y)}\setsep x,y\in U\}$. However, we use this notation exclusively for nondecreasing functions $f$ and intervals $U$, so $\osc(f,U)$ is just the increment of~$f$ over~$U$. The support of $f$ is the set $\supp(f)\eqdef\overline{\{x\in\R\setsep f(x)\neq 0\}}$. We denote the right (resp. left) derivative of $f$ at $x$ by $f'_{+}(x)$ (resp. $f'_-(x)$). We write $\norm{f}_1=\int_\R \abs{f}$ for the $L^1$-norm of $f$; the supremum norm is $\norm{f}_\infty = \sup_{x\in\R}\abs{f(x)}$. Given a set $A\subseteq\R$, the symbol $\Char_A$ denotes the characteristic function of $A$.

We shall be using some well-known facts about absolutely continuous functions, some of which are arranged into the following lemma.

We do not give the detail of its proof; it can be proved by a combination of Fubini's theorem about interchanging the order of summation and differentiation (see for example Theorem~1.4.1 in~\cite{KK96}) and the monotone convergence theorem.
\begin{lemma}\label{sum AC}
  Let $f_k\colon \field{R}\to \field{R}$ be locally absolutely continuous and monotone increasing such that $f\eqdef \sum_{k=1}^\infty f_k$ exists.
  Then the function~$f$ is locally absolutely continuous.

  If moreover $\sum_{k=1}^{\infty}\norm{f_k'}_{1}<\infty$, then $f$ is absolutely continuous.
  A special case is if $\varphi\in L^1$, and $g\colon \field{R}\to \field{R}$ is defined by $g(x)=\int_{-\infty}^x \varphi(t)\dt$, then
  $g$ is absolutely continuous.
\end{lemma}

We shall also need the following simple lemma. We shall be using the term \emph{disjoint $K_\sigma$} (which we abbreviate $DK_\sigma$) for any set that can be expressed as the union of countably many pairwise disjoint compact sets.
\begin{lemma}\label{L:DisjointFSigma}
  Any meagre $F_\sigma$-set in $\R$ is $DK_\sigma$.
\end{lemma}

\begin{proof}
 Let us first make two simple observations, providing a proof only for the second one:
  \begin{itemize}
      \item A countable union of pairwise disjoint $DK_\sigma$-sets is itself $DK_\sigma$.
      \item Let $F\subseteq \R$ be a nowhere dense closed set, and $I=(a,b)\subseteq\R$ ($a,b\in \R\cup\{-\infty,\infty\}$) be an open interval. Then $F\cap I$ is $DK_\sigma$.
  \end{itemize}
  The first observation is obvious. To prove the second one, we use the nowhere denseness of $F$ to find an increasing sequence $(x_n)_{n=-\infty}^\infty$ in $I\setminus F$ such that $\lim_{n\to-\infty}x_n=a$ and $\lim_{n\to\infty}x_n=b$. Then we have the following expression of $F\cap I$, which makes it apparent that $F\cap I$ is $DK_\sigma$: 
  \begin{equation*}
      F\cap I= F\cap (a,b) = \bigcup_{n=-\infty}^\infty \left[x_n,x_{n+1}\right] \cap F.
  \end{equation*}

  Having taken care of the two observations, we take an arbitrary meagre set $B\subseteq \R$ of the type $F_\sigma$. Then $B$ can be written as $\bigcup_{n=1}^\infty F_n$ where all $F_n$ are compact. Moreover, each $F_n$ is also nowhere dense: Indeed, if $F_n$ were not nowhere dense, then it would contain a nontrivial open interval, which would make $B$ nonmeagre by the Baire category theorem. 
  
  We can express $B$ as the following disjoint union
  \begin{equation*}
      B= F_1\cup (F_2\setminus F_1)\cup (F_3\setminus (F_1\cup F_2))\cup\dots = \bigcup_{n=1}^\infty \biggl(F_n\setminus \bigcup_{k=1}^{n-1}F_k\biggr),
  \end{equation*}
  so (by the first observation) it suffices to show, given a natural number $n\geq 2$, that $F_n\setminus \bigcup_{k=1}^{n-1}F_k$ is $DK_\sigma$. To that end, define $\II$ to be the set of all components of $\R\setminus \bigcup_{k=1}^{n-1}F_k$; then the elements of $\II$ are pairwise disjoint open intervals. Now we have
  \begin{equation*}
      F_n\setminus \bigcup_{k=1}^{n-1}F_k = \bigcup_{I\in\II}(F_n\cap I),
  \end{equation*}
  where the union on the right-hand side is clearly disjoint, and each of the sets $F_n\cap I$ is $DK_\sigma$ by the second observation. Hence, the first observation implies $F_n\setminus \bigcup_{k=1}^{n-1}F_k$ to be $DK_\sigma$ as required. The proof is complete.
\end{proof}

\section{Proofs}
In this section, we gradually build towards a proof of \cref{T:main}. Although the longest proof is that of \cref{L:functiong} as we need to prove the function $g$ constructed therein has many particular properties, the main ideas are contained in the proof of \cref{null}, which is essentially the same as the main theorem but contains the extra assumption that $A$ be meagre. Getting rid of meagreness is then a simple task.

\begin{notation}\label{N:FepsFhat}
  Given a closed set $F\subseteq \R$, we denote
  \begin{align*}
    (F)_\eps &= \{x\in \R \setsep \dist(x,F)<\eps\}\qquad\text{and}\\
    \widehat{F} &= F\cup \bigcup_{n=1}^\infty \Bigl\{x\in\R\setsep \dist(x,F)=\frac{1}{n}\Bigr\}.
  \end{align*} 
\end{notation}
In the next lemma, we say that a subset of the real line has everywhere positive measure (EPM) if it its intersection with every nonempty open interval has positive Lebesgue measure.

\begin{lem}\label{L:SetM}
  Let $M\subseteq\R$ have EPM. Then there are disjoint subsets $M_1,M_2\subseteq M$, both having EPM.
\end{lem}
\begin{proof}
  Let $(I_n)_{n=1}^\infty$ be a basis of open sets in $\R$ consisting of open intervals. By the regularity of the Lebesgue measure, we may choose disjoint compact sets $K_1,L_1 \subseteq I_1\cap M$ of positive Lebesgue measure; we can also assume them to be nowhere dense: indeed, if e.g. $K_1$ were not, then it would contain a nontrivial interval $J$ as it is closed. We would then replace $K_1$ by a ``fat Cantor set'' contained in $J$.

  Now, assume that the nowhere dense compacta $K_1,L_1,\dots, K_n,L_n$ have already been constructed. Then $K\eqdef \bigcup_{i=1}^n (K_i\cup L_i)$ is nowhere dense, and so $I_{n+1}\setminus K$ contains a nonempty open interval, say $\tilde{I}_{n+1}$. Again, choose disjoint nowhere dense compact sets $K_{n+1}, L_{n+1}\subseteq \tilde{I}_{n+1}\cap M$ of positive measure.

  Setting $M_1=\bigcup_{n=1}^\infty K_n$ and $M_2=\bigcup_{n=1}^\infty L_n$, it is easy to observe that $M_1,M_2\subseteq M$, we have $M_1\cap M_2=\emptyset$, and that both sets have EPM.
\end{proof}

\begin{lem}\label{L:SetH}
  Suppose $F\subseteq \R$ is closed and $M\subseteq \R$ is a measurable set that meets every interval in a set of positive measure. Then for every $\eps>0$ there is a closed set $H$ such that
  \begin{enumerate}[(1)]
    \item\label{LI:Set1} $F\subseteq H\subseteq (F\cup M)\cap (F)_\eps$;
    \item\label{LI:Set2} $H$ meets the middle third of every component of $\widehat{F}^c\cap (F)_\eps$ in a set of positive measure.
  \end{enumerate}
\end{lem}

\begin{proof}
  Assume $F\neq\emptyset$; the statement is trivial otherwise. The set $\widehat{F}$ is easily seen to be closed, so if $J$ is an arbitrary component of~$\widehat{F}^c\cap (F)_\eps$, it is an open interval.
  Given any such $J=B(c,r)$, the regularity of the Lebesgue measure permits us to choose a compact set~$H_J$ with $\abs{H_J}>0$ and
  \begin{equation}\label{HIproperty}
    H_J\subseteq B\Bigl(c,\frac{r}{3}\Bigr)\cap M\subseteq J\cap M;
  \end{equation}
  in particular, $H_J$ is contained in the middle third of $J$. We define
  \begin{equation}\label{Def H}
    H\eqdef F\cup\bigcup_J H_J,
  \end{equation}
  where the union is over all components~$J$ of~$\widehat{F}^c\cap (F)_\eps$;
  then \cref{LI:Set1}~and~\cref{LI:Set2} are obviously satisfied (by the choice of $H_I$).
 
  We are left to show that $H$ is closed.
Pick an arbitrary $x\notin H$. Then $x\notin F$, so there exists a component $(a,b)$ of $F^c$ (with $a,b\in\R\cup\{-\infty,\infty\}$) containing $x$.
We set $\delta_0\eqdef\frac{1}{2}\dist(x,F)$, which is positive as $F$ is closed, and we have 
\begin{equation*}
    a<a+\delta_0\leq x-\delta_0<x<x+\delta_0\leq b-\delta_0<b.
\end{equation*}
Let $\JJ$ be the family of all the components of~$\widehat{F}^c\cap (F)_\eps$ that meet the interval $(a+\delta_0,b-\delta_0)$. Then clearly
\begin{equation}\label{E:HisClosedLemma}
  (a+\delta_0,b-\delta_0)\cap \biggl(F\cup \bigcup_{J\notin \JJ}H_J\biggr) = \emptyset.
\end{equation}
It is easy to see that $\JJ$ is finite (including the cases when $a$ or $b$ is infinite). Hence,
\begin{equation*}
  C\eqdef\bigcup_{J\in \JJ}H_J
\end{equation*}
is closed. 
Since $x\notin H$ and $C\subseteq H$, we have $x\notin C$ implying that $\delta_1\eqdef\dist(x,C)>0$.
We set $\delta\eqdef\min\{\delta_0,\delta_1\}$; then $B(x,\delta)\cap C=\emptyset$. Together with \cref{E:HisClosedLemma}, this shows that $B(x,\delta)\cap H=\emptyset$, concluding the proof.
\end{proof}

\begin{lem}\label{L:functiong}
  Suppose $E,F\subseteq \R$ are closed, $A,M\subseteq \R$ are measurable, $\abs{A}=0$ and $M$ meets every interval in a set of positive measure. 
  Let $\eps>0$ and $H$ be the closed set from~\cref{L:SetH}:
  \begin{enumerate}[(1)]
    \item $F\subseteq H\subseteq (F\cup M)\cap (F)_\eps$;
    \item $H$ meets the middle third of every component of $\widehat{F}^c\cap (F)_\eps$ in a set of positive measure.
  \end{enumerate}  
  Then there is a nondecreasing absolutely continuous function $g\colon \R\to [0,\eps]$ such that
  \begin{enumerate}[(i)]
    \item\label{LI:1} $\spt(g')\subseteq H$;
    \item\label{LI:2} $\Lip g(x)<\infty$ for every $x\in\R$;
    \item\label{LI:3} $g'(x)=1$ for every $x\in A\cap F\cap E^c$;
    \item\label{LI:4} $\osc(g,U)\leqslant \eps\abs{U}$ whenever $U$ is an interval meeting $E$;
    \item\label{LI:5} $\norm{g'}_{1}<\eps$.
  \end{enumerate}

\end{lem}

\begin{proof}
  Without loss of generality we may assume that all components of $E^c$ are bounded: Indeed, since $A$ has measure zero, we may find a strictly increasing sequence $(x_k)_{k=-\infty}^{\infty}$ without accumulation points such that none of the $x_k$'s lies in~$A$ and $\lim_{k\to -\infty}x_k=-\infty$ and $\lim_{k\to \infty}x_k=\infty$. It is now clear from the statement of the lemma, that if we prove it with $E$ replaced by the (obviously closed) set $E\cup \{x_k\setsep k\in \Z\}$, it will also be proved for $E$. Hence, there is no loss in generality in the assumption, which we shall adopt, that the complement of~$E$ is a countable union of bounded open intervals.

  Let $\II$ be the collection of all components $I$ of $E^c$. Denote $\tilde{A}=A\cap F\cap E^c$, and let $G_0\subseteq\R$ be an open set such that $\tilde{A}\subseteq G_0\subseteq (F)_\eps$ and $\abs{G_0}<\eps$. 

  Pick any $I\in \II$; then $I=(a,b)$ is a bounded open interval. Let us choose a sequence $\left( G^n_I \right)_{n=1}^\infty$ of open subsets of $I$
  satisfying, for every $n\in\N$, the following conditions (as  $\abs*{\widetilde{A}}=0$, we can find such sets):
  \begin{itemize}
    \item\label{E:GnIa} $\tilde{A}\cap I\subseteq G^n_I \subseteq G_0\cap I$;
    \item $\bigl(a+\frac{\abs{I}}{n}, b-\frac{\abs{I}}{n}\bigr)\subseteq G^n_I$;
    \item $\abs{G^n_I\cap \bigl(a,a+\frac{\abs{I}}{n}\bigr)}<\frac{\eps}{4}\cdot \frac{\abs{I}}{n+1}$;
    \item $\abs{G^n_I\cap \bigl(b-\frac{\abs{I}}{n},b\bigr)}<\frac{\eps}{4}\cdot \frac{\abs{I}}{n+1}$.
  \end{itemize}
  Set $G_I\eqdef \bigcap_{n=1}^{\infty} G^n_I$; we show that $G_I$ is open: Indeed, given any $x\in G_I\subseteq I = (a,b)$, there clearly exist $\delta_1>0$ and $n_0\in\N$ such that for all $n\geq n_0$ we have $B(x,\delta_1)\subseteq \bigl(a+\frac{\abs{I}}{n}, b-\frac{\abs{I}}{n}\bigr) \subseteq G^n_I$. Of course, the sets $G^n_I$, $n\in\{1,\dots, n_0-1\}$, are all open and contain $x$, so there is some $\delta_2>0$ such that  $B(x,\delta_2)\subseteq \bigcap_{n=1}^{n_0-1}G^n_I$, and it follows that $B(x,\min\{\delta_1,\delta_2\})\subseteq G_I$.
  \medskip

  \emph{Claim:} For any interval $U\subseteq I=(a,b)$ with $\partial U\cap \partial I\neq\emptyset$ we have
  \begin{equation}\label{E:Claim0}
      \abs{G_I\cap U}<\frac{\eps}{4}\abs{U}.
  \end{equation}

  For the proof of the claim, suppose $U=(a,c)$ for some $c\in(a,b]$; the other case can be dealt with using a symmetric argument. Take the unique $n\in \N$ with
  \begin{equation*}
    a+\frac{\abs{I}}{n+1}<c\leq a+\frac{\abs{I}}{n}.
  \end{equation*} 
  The following estimates prove the claim:
  \begin{dmath*}
    \abs{G_I\cap U} \leq \abs{G^n_I\cap U}\leq \abs{G^n_I\cap \Bigl(a,a+\frac{\abs{I}}{n}\Bigr)} <\frac{\eps}{4}\cdot\frac{\abs{I}}{n+1}<\frac{\varepsilon}{4}(c-a)=\frac{\eps}{4}\abs{U}.    
  \end{dmath*}%
  \medskip

  Having finished the above construction for every $I\in\II$, we now set $G\eqdef \bigcup_{I\in\II} G_I$. Note that 
  \begin{equation}\label{E:InclusionsAG}
    \tilde{A}\subseteq G\subseteq E^c\cap G_0\subseteq (F)_\eps 
  \end{equation}  
  and $\abs{G}<\eps$; indeed, $\tilde{A}\cap I\subseteq G_I$ for all $I\in\II$, $\tilde{A}\subseteq E^c$, and $E^c=\bigcup \II$; moreover, $\abs{G_0}<\eps$ and $G_0\subseteq (F)_\eps$. 
  
  Next we set (following \cref{N:FepsFhat})
  \begin{equation}\label{E:DefFHat}
    \widehat{F}\eqdef F\cup \bigcup_{n=1}^\infty \Bigl\{x\in\R\setsep \dist(x,F)=\frac{1}{n}\Bigr\}.
  \end{equation}  
  Clearly, $\widehat{F}$ is closed as $F$ is closed. Further, we define 
  \begin{equation}\label{E:DefOfJ}
    \JJ\text{ to be the set of all components $J$ of }\bigl(\widehat{F}\bigr)^c\text{ with }J\subseteq G;    
  \end{equation}
  observe that the set $\hat{F}\cup\bigcup\JJ$ contains $\tilde{A}$ in its interior. Further, each $J\in\JJ$ is contained in some $I\in\II$ as $J$ is connected and $J\subseteq G\subseteq E^c$, so it is contained in some component of $E^c$. Now, for every interval $J=B(c,r)\in\JJ$, choose a bounded measurable function $\varphi_J\geq 0$ supported in $H\cap B\bigl(c,\frac{r}{3}\bigr)$ such that $\int_J \varphi_J = \abs{J}$; this is possible thanks to the fact that $H$ meets the middle third of every $J\in\JJ$ in a set of positive measure, as follows from \cref*{L:SetH}~\cref{LI:Set2} (note that every $J\in\JJ$ satisfies $J\subseteq G\subseteq (F)_\eps$). 
  Remembering that the sets $G_I$ are pairwise disjoint and that their union is by definition~$G$, we
  define
  \begin{equation}\label{E:DefPhiG}
    \varphi\eqdef \sum_{J\in\JJ}\varphi_J + \sum_{I\in\II}\Char_{F\cap G_I}=\sum_{J\in\JJ}\varphi_J +\Char_{F\cap G}\quad\text{ and }\quad g(t)\eqdef \int_{-\infty}^{t}\varphi.
  \end{equation}

We show $g$ has the desired properties, starting with~\cref{LI:5}. First we observe that $\varphi$ is integrable; indeed, all the summands in the definition of $\varphi$ are nonnegative, so the monotone convergence theorem yields \begin{equation*}
    \int_\R \varphi = \sum_{J\in\JJ}\int_\R\varphi_J + \int_\R \Char_{F\cap G} = \sum_{J\in\JJ} \abs{J} + \abs{F\cap G}\leq \abs{G},
\end{equation*}
where the last inequality follows from the fact that the intervals $J\in\JJ$, together with $F\cap G$, are pairwise disjoint subsets of $G$. As $\varphi$ is nonnegative,  $g$ is nondecreasing. Since $\abs{G}<\eps$, we have $g\colon\R\to [0,\eps]$, as required.
Moreover, \cref{sum AC} implies that $g$ is absolutely continuous. Therefore, $g'=\varphi$ almost everywhere, whence $\norm{g'}_1=\int_\R\abs{g'}= \int_\R\abs{\varphi}=\int_\R\varphi\leq \abs{G}<\eps$, implying \cref{LI:5}.%
\medskip

We show \cref{LI:1}. By their definitions, all of the summands in~\cref{E:DefPhiG} are supported in the closed set $H$. Therefore, given any $x\notin H$, there exists $\delta>0$ such that $\varphi=0$ in $B(x,\delta)$; that obviously makes $g$ constant on $B(x,\delta)$, whence $x\notin\spt(g')$. Thus is proved \cref{LI:1}.%
\medskip

To prove \cref{LI:2}, we pick an arbitrary $x\in\R$. If $x\notin H$, then $g$ is constant on a neighbourhood of $x$ by the above, so $\Lip g(x)=0$; hence, we assume $x\in H$. We divide the task into two parts by setting
\begin{align*}
    g_1(t) &\eqdef \int_{-\infty}^t \sum_{J\in\JJ} \varphi_J, \\
    g_2(t) &\eqdef \int_{-\infty}^t \Char_{F\cap G};
\end{align*}
we then have $g=g_1+g_2$ and since $\Lip g(x)\leq \Lip g_1(x) + \Lip g_2(x)$, it suffices to prove that $\Lip g_1(x)$ and $\Lip g_2(x)$ are both finite. 

We start with $g_1$; recall that we are in the case where $x\in H$. If $x\in J$ for some $J\in\JJ$, then for any $y\in J$ we have 
\begin{equation*}
    \abs{g_1(y)-g_1(x)}=\abs{\int_x^y \varphi_J} \leq \abs{y-x}\norm{\varphi_J}_{\infty}.
\end{equation*}
As $\varphi_J$ is bounded, this shows that $\Lip g_1(x)\leq \norm{\varphi_J}_{\infty}<\infty$.

If, on the other hand, $x\notin J$ for any $J\in\JJ$, then we observe what happens on either side (i.e.\ left or right) of $x$. If there is a $\delta>0$ such that $(x,x+\delta)\cap \bigcup\JJ=\emptyset$, then we obviously get $\Lip g_1(x_+)=0$. Hence, let us assume $x$ is approximated from the right by intervals in $\JJ$, and take an arbitrary $y>x$.

Assume first that $y\notin \bigcup \JJ$.
Taking the sum over all intervals $J\in\JJ$ that are contained in $(x,y)$, we have 
\begin{equation}\label{E:PfFction1}
  \abs{g_1(y)-g_1(x)} = \sum \abs{J} \leq \abs{y-x}.  
\end{equation}

If, on the other hand, there is $J\in\JJ$ with $y\in J$, then there are two cases: 
\begin{enumerate*}[(a)]
    \item\label{EI:PfFction:1} $y$ is in the open left third of $J$, or
    \item\label{EI:PfFction:2} not.
\end{enumerate*}
Let $z=\inf J$; by the assumption on $x$, we have $x\leq z$.

Assume \cref{EI:PfFction:1} is the case.  As $\varphi_J$ is supported in the middle third of~$J$, we see that $g_1(z)=g_1(y)$. Using this and the estimate \cref{E:PfFction1} with $y$ replaced by $z$, we obtain
\begin{equation*}
    \abs{g_1(y)-g_1(x)}=\abs{g_1(z)-g_1(x)}\leq \abs{z-x}<\abs{y-x}.
\end{equation*}

If the case that occurs is \cref{EI:PfFction:2}, then $\abs{y-z}\geq \frac{\abs{J}}{3}$, and clearly $\abs{g_1(y)-g_1(z)}\leq \abs{J}$ by the choice of $\varphi_J$; hence
\begin{equation*}
    \abs{g_1(y)-g_1(z)}\leq 3\abs{y-z}.
\end{equation*}
Moreover, (again by \cref{E:PfFction1}) we have $\abs{g_1(z)-g_1(x)}\leq \abs{z-x}$. It follows that
\begin{align*}
    \abs{g_1(y)-g_1(x)} &\leq\abs{g_1(y)-g_1(z)}+\abs{g_1(z)-g_1(x)}\\
    &\leq  3\abs{y-z}+\abs{z-x}  \leq 3\abs{y-z} + 3\abs{z-x} = 3\abs{y-x}.
\end{align*}
Thus we conclude that if $x\notin \bigcup\JJ$, then $\Lip g_1(x_+)\leq 3$. Similarly, we obtain $\Lip g_1(x_-)\leq 3$, and so, again, $\Lip g_1(x)<\infty$. 

As for $g_2$, by its definition we have, for all $x,y\in\R$,
\begin{equation*}
    \abs{g_2(y)-g_2(x)} = \abs{\int_x^y \Char_{F\cap G}}\leq \abs{y-x},
\end{equation*}
so $g_2$ is $1$-Lipschitz. Thus we obtain $\Lip g(x) = \Lip(g_1+g_2)(x) \leq \Lip g_1(x) + \Lip g_2(x)< \infty$, and  \cref{LI:2}  is proved.%
\medskip

Instead of proving directly the full version of~\cref{LI:3}, we only aim to prove that the right derivative equals~$1$ at any point of $\widetilde{A}=A\cap F\cap E^c$, which easily follows from~Claim~3 below.
The left derivative can be handled analogously.
Our strategy shall be to prove Claim~3 below; we then apply this result, almost immediately obtaining $g'_+(x)=1$ for any $x\in \widetilde{A}$. We shall need to distinguish several cases based on the choice of $y$. In Claim~1, we start with the simplest one, $y\in\widehat{F}$, which will later be useful multiple times. We then move on to Claim~2, proving the estimate under the assumption $(x,y)\cap F=\emptyset$. The proof of Claim~3 consists in applying both Claims~1 and~2 to treat the remaining case $(x,y)\cap F\neq\emptyset$. We keep in mind, in particular, the definitions of $\hat{F}$, $\JJ$ and $g$ (see \cref{E:DefFHat}, \cref{E:DefOfJ}, and \cref{E:DefPhiG}).%
\medskip


\emph{Claim 1:} For any $x\in F\cap G$, and any $y\in \widehat{F}\cap(x,\infty)$ such that $[x,y]\subseteq G\cap(\widehat{F}\cup\bigcup\JJ)$, we have $g(y)-g(x)=y-x$.

To prove the claim, choose arbitrary $x,y$ as in the statement. Let us note that $\widehat{F}\setminus F$ is countable ($F^c$ has countably many components and $\widehat{F}$ is countable in each of them), and
\begin{equation}\label{E:intervalContained}
  (x,y)\subseteq F\cup \bigl(\widehat{F}\setminus F\bigr) \cup \bigcup_{(\ast)}J
\end{equation}
where by $(\ast)$ we mean  all $J\in\JJ$ with $J\subseteq (x,y)$ (recall that $y\in\widehat{F}$). Explaining the individual equalities just below, we perform the following computation:
\begin{equation} \label{E:FHatOK}
\begin{aligned}
    g(y)-g(x) &= \int_x^y \Char_{F\cap G} + \int_x^y \sum_{J\in\JJ} \varphi_J \\
    &= \int_x^y \Char_F + \sum_{(\ast)} \int_x^y \varphi_J\\
    &= \abs{F\cap (x,y)}+\sum_{(\ast)}\abs{J} \\
    &= \abs{(x,y)}=y-x.
\end{aligned}
\end{equation}
The first equality is from the definition of $g$, the second equality holds as $(x,y)\subseteq G$ and by the monotone convergence theorem, and the fourth equality is by \cref{E:intervalContained}. Claim~1 is proved.%
\medskip

\emph{Claim 2:}
For any $x\in G\cap F$, any $y>x$ with $[x,y]\subseteq G\cap (\widehat{F}\cup\bigcup\JJ)$ and $(x,y)\cap F=\emptyset$, and any $N\geq 2$ with $y<x+1/N$, we have 
\begin{equation}\label{E:RightIsolatedEst}
   \frac{g(y)-g(x)}{y-x} \in \left(\frac{N-1}{N+1}, \frac{N+1}{N-1} \right). 
\end{equation}


In order to prove the claim, we pick arbitrary $x,y,N$ as in the statement. We may further assume $y\notin \widehat{F}$ as Claim~1 clearly covers the opposite possibility.
Let $J=(u,v)\in \JJ$ be the one element satisfying $y\in J$; it exists as $y\in\bigcup\JJ$. Then $J$ is bounded because $y\in(x,x+1/N)\subseteq(x,x+1/2)$. As $u\in \widehat{F}$, we have $g(u)-g(x)=u-x$ by Claim~1 (where we replace $y$ by $u$). Therefore we obtain
\begin{equation*}
    g(y)-g(x) = g(y)-g(u) + g(u)-g(x) = g(y)-g(u)+u-x.
\end{equation*}
Since $g(y)-g(u)=\int_u^y \varphi_J \in \left[0,\abs{J}\right]$ by the choice of $\varphi_J$, it follows that
\begin{equation*}
    g(y)-g(x) \in \Bigl[u-x, u-x+\abs{J} \Bigr] = \bigl[u-x, v-x \bigr],
\end{equation*}
from where we infer, using that $y\in J=(u,v)$,
\begin{equation}\label{E:UnivEst}
  \frac{g(y)-g(x)}{y-x} \in \left( \frac{u-x}{v-x}, \frac{v-x}{u-x} \right).
\end{equation}

Now, if $y$ is in an unbounded component of $F^c$ (i.e.\ $(x,\infty)\cap F=\emptyset$), then the interval $J\in \JJ$ containing $y$ is of the form $(u,v)=\left(x+\frac{1}{n+1},x+\frac{1}{n}\right)$ for some natural number~$n$; clearly, $n\geqslant N$ as $y<x+1/N$. But this implies
\begin{equation}\label{E:UnbddEst}
  \left( \frac{u-x}{v-x}, \frac{v-x}{u-x} \right)=\left(\frac{n}{n+1}, \frac{n+1}{n} \right)\subseteq \left(\frac{N}{N+1}, \frac{N+1}{N}\right).
\end{equation}
Of course, \cref{E:UnivEst} and \cref{E:UnbddEst} yield
\begin{equation}\label{E:EasyEst}
    \frac{g(y)-g(x)}{y-x} \in \left(\frac{N}{N+1}, \frac{N+1}{N}\right).
\end{equation}


On the other hand, if $y$ is in a bounded component of $F^c$, say $(x,b)$, then the distribution of $\widehat{F}$ in $(x,b)$ is symmetric, and the $I\in \JJ$ with $I\subseteq (x,b)$ are of three types:
\begin{enumerate}[(a)]
  \item\label{LI:CaseA} $I$ is contained in the left half of $(x,b)$;
  \item\label{LI:CaseB} $I$ is contained in the right half of $(x,b)$;
  \item\label{LI:CaseC} $I$ contains the centre of $(x,b)$. 
\end{enumerate}
If $J$ is of type \cref{LI:CaseA}, then $J=(u,v)=\left(x+\frac{1}{n+1},x+\frac{1}{n}\right)$, and we obtain \cref{E:UnbddEst}, and subsequently \cref{E:EasyEst}, in the same way as above.
If we have~\cref{LI:CaseB}, by the symmetry of $\widehat{F}$ in $(x,b)$, the interval $J=(u,v)$ has a symmetric counterpart $J'=(u',v')$ in the left half of $(x,b)$. Set $d\eqdef u-u'>0$ (then $d=v-v'$). Now we have
\begin{align*}
    \left(\frac{u-x}{v-x}, \frac{v-x}{u-x} \right) &= \left(\frac{u'-x+d}{v'-x+d}, \frac{v'-x+d}{u'-x+d} \right) \\
    &\subseteq \left(\frac{u'-x}{v'-x}, \frac{v'-x}{u'-x} \right) \\
    &\subseteq \left(\frac{N}{N+1}, \frac{N+1}{N}\right)
\end{align*}
where the last inclusion follows from case \cref{LI:CaseA}: Recall that $J'=(u',v')$ is in the left half of $(x,b)$, so \cref{E:UnbddEst} is satisfied with $u,v$ replaced by $u',v'$, respectively. By virtue of \cref{E:UnivEst} we, again, obtain \cref{E:EasyEst}.

Finally, if \cref{LI:CaseC} occurs, we have $J=\left(x+\frac{1}{n},b-\frac{1}{n}\right)$ for some $n\in\N$. Then  $n>N$ as $1/n<y-x<1/N$. Let us now observe that $x+\frac{1}{n-1}$ is not in the left half of $(x,b)$ 
as otherwise we would have $x+\frac{1}{n-1}\in \widehat{F}$ and $x+\frac{1}{n-1}\in\left(x+\frac{1}{n},b-\frac{1}{n}\right)=J$, so $J$ would not be a component of $\widehat{F}^c$, a contradiction. Since $J$ shares its centre with $(x,b)$, it follows that $x+\frac{1}{n-1}$ is not in the left half of $J=\left(x+\frac{1}{n},b-\frac{1}{n}\right)$ either, and thus we obtain
\begin{equation*}
    \frac{\abs{J}}{2} \leqslant\left(x+\frac{1}{n-1}\right)-\left(x+\frac{1}{n}\right)=\frac{1}{n(n-1)}.
\end{equation*}
Denoting $J=(u,v)$, we therefore have $u-x=\frac{1}{n}$ and $v-x=u-x+\abs{J}=\frac{1}{n}+\abs{J}\leq\frac{1}{n}+\frac{2}{n(n-1)}$, whence
\begin{equation*}
    \frac{u-x}{v-x} \geq \frac{\frac{1}{n}}{\frac{1}{n}+\frac{2}{n(n-1)}}=\frac{1}{1+\frac{2}{n-1}}= \frac{n-1}{n+1}.
\end{equation*}
Using this and \cref{E:UnivEst}, we immediately derive \cref{E:RightIsolatedEst}:
\begin{equation}\label{E:MiddleInt}
   \frac{g(y)-g(x)}{y-x} \in \left(\frac{n-1}{n+1}, \frac{n+1}{n-1} \right)\subseteq \left(\frac{N-1}{N+1}, \frac{N+1}{N-1} \right). 
\end{equation}
We summarize our findings by stating that in every possible case we either have \cref{E:FHatOK}, or \cref{E:EasyEst}, or \cref{E:MiddleInt}, and that each of these propositions implies \cref{E:RightIsolatedEst}. The claim is proved.%
\medskip

\emph{Claim 3:} For any $x\in G\cap F$, any $y>x$ with $[x,y]\subseteq G\cap(\widehat{F}\cup\bigcup\JJ)$, and any $N\geq 2$ with $y<x+1/N$, we have \cref{E:RightIsolatedEst}.


For the proof of the claim, we again pick any $x, y, N$ satisfying the assumptions of our claim, and by the preceding claims we may further assume that $y\notin F$ and $(x,y)\cap F\neq \emptyset$.

Set $z=\max((x,y)\cap F)$ and let us check the assumptions of Claim~2 for the interval $[z,y]$, i.e.\ with $x$ replaced by $z$: By the choice of $z$, we have $z\in F$ and $(z,y)\cap F=\emptyset$. Moreover, $[z,y]\subseteq [x,y]\subseteq G\cap(\widehat{F}\cup\bigcup\JJ)$, the second inclusion being one of the assumptions of the present claim, and we also have $y<x+1/N<z+1/N$. 

By Claim~2 (with $x$ replaced by $z$) we thus have 
\begin{equation*}
  \frac{g(y)-g(z)}{y-z} \in \left(\frac{N-1}{N+1}, \frac{N+1}{N-1} \right).     
\end{equation*}
and by Claim~1 (with $y$ replaced by $z$, and noting that $z\in F\subseteq \widehat{F}$), 
\begin{equation*}
  \frac{g(z)-g(x)}{z-x}=1.
\end{equation*}
Since the average slope of $g$ over $[x,y]$ (i.e., the quotient $(g(y)-g(x))/(y-x)$) is a convex combination of the average slopes over $[x,z]$ and $[z,y]$, we immediately see that \cref{E:RightIsolatedEst} holds also in the present case, concluding the discussion of all possible cases and thus also the proof of the claim.%
\medskip

To conclude the proof of \cref{LI:3}, pick any $x\in\widetilde{A}=A\cap F\cap E^c$ and any $\eta>0$; by \cref{E:InclusionsAG}, $x\in G\cap F$. Recall that $\JJ$ is the set of all components of $\widehat{F}^c$ contained in $G$, which easily implies that $\widehat{F}\cup\bigcup\JJ$ contains $\tilde{A}$ in its interior. Therefore there is $N\geq 2$ such that
\begin{align*}
    \left[x,x+\frac{1}{N}\right] &\subseteq G\cap\left(\widehat{F}\cup\bigcup\JJ\right)\quad\text{and}\\
    \left(\frac{N-1}{N+1},\frac{N+1}{N-1}\right) &\subseteq (1-\eta,1+\eta).
\end{align*}
Obviously, for any $y\in [x,x+1/N]$ we now have $[x,y]\subseteq G\cap (\widehat{F}\cup\bigcup\JJ)$, whence, by virtue of Claim~3, 
\begin{equation*}
    \frac{g(y)-g(x)}{y-x} \in \left(\frac{N-1}{N+1},\frac{N+1}{N-1}\right)\subseteq (1-\eta,1+\eta).
\end{equation*}
The proof of \cref{LI:3} is finished.%
\medskip

Finally, we prove \cref{LI:4}, which states that, for any interval $U\subseteq \R$ with $U\cap E\neq\emptyset$, we have 
\begin{equation}\label{E:ContentOfLI4}
  \osc(g,U)\leq \eps\abs{U}.    
\end{equation}
Recall that $\eps>0$ was fixed already in the statement of the lemma, $E$ is closed, and $\II$ is the collection of all components of $E^c$. Recalling \cref{E:DefOfJ}, we see that the elements of $\JJ$ are pairwise disjoint open intervals, each contained in some $I\in\II$.

We start by choosing an arbitrary $I=(a,b)\in\II$ and any open interval $U\subseteq I$ with $\partial U\cap\partial I\neq \emptyset$; our first aim shall be to prove \cref{E:ContentOfLI4} in this case. (Note that $\osc(g,U)=\osc(g,\overline{U})$ and that, unlike $U$, the closed interval $\overline{U}$ meets $E$.) There is no loss of generality in assuming that $U$ shares with $I$ its left endpoint, i.e.\ $U=(a,c)$ for some $c\in(a,b]$; the other case can be dealt with analogously. 

Keeping in mind the definitions of $\varphi$ and $g$ (cf. \cref{E:DefPhiG}), in particular the fact that $\varphi\geq 0$ so $g$ is monotone, we compute (explanations of the last three steps and the corresponding notations follow just afterwards):
\begin{equation}\label{E:ComputationLI4}
\begin{aligned}
  \osc(g,U) &= g(\sup U)- g(\inf U) =\int_U \varphi\\
  &= \int_U\biggl( \sum_{J\in\JJ} \varphi_J + \Char_{G\cap F}\biggr)\\
  &= \abs{G\cap F\cap U} + \sum_{J\in\JJ}\int_U \varphi_J \\
  &= \abs{G_I\cap F\cap U} + \sum_{(\ast)}\abs{J} + \sum_{(\ast\ast)} \int_U\varphi_J\\
  &\leq \abs{G_I\cap U} + \sum_{(\ast\ast)} \int_U\varphi_J \\
  &< \frac{\eps}{4}\abs{U} +\sum_{(\ast\ast)} \int_U\varphi_J.
\end{aligned}
\end{equation}
Here by $(\ast)$ we mean all the intervals $J\in\JJ$ with $J\subseteq U$. By $(\ast\ast)$ we mean all the intervals $J\in\JJ$ with $J\cap U\neq\emptyset$ and $J\nsubseteq U$. The fifth equality is a consequence of the following facts: $U\subseteq I$ and $G\cap I=G_I$, whence $G\cap U=G_I\cap U$; for any $J\in\JJ$, $\int_{\R} \varphi_J = \int_J \varphi_J = \abs{J}$. The next inequality follows from the fact that all $J\in\JJ$ are pairwise disjoint, disjoint from $F$, and the ones pertaining to $(\ast)$ are contained in $G_I\cap U$. The final inequality is from \cref{E:Claim0}.

Since each $J\in \JJ$ is contained in some $I\in \II$ and we are looking at such an~$I$ sharing its left endpoint with the one of~$U$, it is easy to see that $(\ast\ast)$ represents at most one interval, so the sum has at most one summand.
If the sum $\sum_{(\ast\ast)}$ in the above computation is empty, then we are done. So let us assume there is $K\in\JJ$ such that $K\cap U\neq\emptyset$ and $K\nsubseteq U$. Then, clearly, $K=(u,v)\subseteq I$ with $u<c<v$. There are two cases: $\abs{K}\leq 3\abs{U}$ or $\abs{K}>3\abs{U}$.

We consider the case $\abs{K}>3\abs{U}$, first. We have intervals $U,K\subseteq I$ such that $K$ is more than three times the length of $U$ and $U$ shares its left endpoint with $I$. From this it is obvious that $U\cap K$ is contained in the (open) left-hand side third of the interval $K$, and so (by the choice of $\varphi_K$) we have $\int_U \varphi_K=0$. That is, the sum $\sum_{(\ast\ast)}$ consists of one summand whose value is $0$, and we are, again, done by \cref{E:ComputationLI4}.

Suppose $\abs{K}\leq 3\abs{U}$ and set $U_1\eqdef U\cup K = (a,v)$, and replace $U$ by $U_1$ in computation \cref{E:ComputationLI4} and adapt the meaning of $(\ast)$ and $(\ast\ast)$; then the sum $\sum_{(\ast\ast)}$ becomes empty. We now obtain the desired estimate as follows:
\begin{gather*}
    \osc(g,U)\leq \osc(g,U_1)<\frac{\eps}{4}\abs{U_1} +0 \\ 
    \leq \frac{\eps}{4}(\abs{K}+\abs{U})\leq \frac{\eps}{4}(3\abs{U}+\abs{U})=\eps\abs{U}.
\end{gather*}%
\medskip

We have proved \cref{E:ContentOfLI4} for any $I\in\II$ and any open interval $U\subseteq I$ with $\partial U\cap \partial I\neq\emptyset$. Remembering that $E$ is closed and $E^c=\bigcup\II$ reveals that we have actually proved \cref{E:ContentOfLI4} for any interval $U \subseteq E^c$ satisfying $\partial U\cap E\neq \emptyset$. 

Now, let $U$ be any interval meeting $E$. Remembering that $G$ is contained in~$E^c$, it follows immediately from the definitions that $\varphi=0$ on $E$, and obviously 
\begin{equation*}
  U=(U\cap E)\cup (U\cap E^c)= (U\cap E)\cup \bigcup_{(+)}(U\cap I),
\end{equation*}
where $(+)$ represents all $I\in\II$ with $U\cap I\neq\emptyset$. Using these observations we get
\begin{equation*}
    \osc(g,U)=\int_U\varphi = \int_{U\cap E}\varphi + \sum_{(+)}\int_{U\cap I}\varphi = \sum_{(+)}\int_{U\cap I}\varphi.
\end{equation*}
For every $I\in \II$, we set $U_I\eqdef U\cap I$; then each nonempty $U_I$ is an interval. Let us observe that for any $I\in\II$ with $(+)$ (i.e.\ $U_I\neq\emptyset$) we have $\partial U_I\cap\partial I\neq\emptyset$: Indeed, we assumed $U\cap E\neq\emptyset$, i.e.\ $U\nsubseteq E^c$, and in particular, no $I\in\II$ contains $U$. Hence, for any $I\in\II$ meeting $U$, the set $U$ contains a point outside of $I$, and by its convexity, it also contains a point of $\partial I$. 

We have just checked that each nonempty $U_I$ shares an endpoint with~$I$, so by the first part of the proof of \cref{LI:4}, we get $\osc(g,U_I)\leq\eps \abs{U_I}$. Using these observations, we may resume the above computation as follows:
\begin{equation*}
    \osc(g,U)=\sum_{(+)}\int_{U_I}\varphi =\sum_{(+)} \osc(g,U_I)\leq \eps \sum_{(+)} \abs{U_I} \leq \eps\abs{U}.
\end{equation*}
This shows \cref{E:ContentOfLI4} for any interval $U$ meeting $E$, and thus it concludes the proof of \cref{LI:4}, and the lemma.
\end{proof}

\begin{lem}\label{null}
Suppose $A\subseteq \R$ is $F_{\sigma\delta}$, meagre, and Lebesgue null. Then there is a nondecreasing absolutely continuous function $g\colon \R\to\R$ such that $g'(x)=\infty$ for every $x\in A$
and $\lip g(x)<\infty$ for every $x\notin A$.
\end{lem}

\begin{proof}
As the zero function takes care of the case where $A=\emptyset$, we assume that $A\not=\emptyset$.
We claim that there are nonempty compact sets $F_k\subseteq \R$ ($k\in\N$) such that
\begin{enumerate}[(a)]
  \item\label{C:a} $F_0=\R$;
  \item\label{C:b} for all $k,l\in\N$, if $k>l$ and $F_k\cap F_l\neq\emptyset$, then $F_k\subseteq F_l$;
  \item\label{C:c} $A=\bigcap_{N\in\N} \bigcup_{k\geqslant N} F_k$, that is, $A$ is precisely the set of points $x\in\R$ belonging to infinitely many $F_k$'s.
\end{enumerate}

Indeed, the meagreness of $A$ implies the existence of an $F_\sigma$ meagre set $B$ containing $A$. Since $A$ is $F_{\sigma\delta}$, there exist $F_{\sigma}$ sets~$B_n$ such that $A=\bigcap_{n=1}^\infty B_n$. By intersecting each $B_n$ with $B$, we may, and will, assume that $B_n$ is meagre for every $n$. Moreover, we may assume that the sets $B_n$ are nested as the intersection of any two $F_\sigma$-sets is $F_\sigma$; that is, we have $B_n\supseteq B_{n+1}$ for every $n\in\N$. By \cref{L:DisjointFSigma}, each $B_n$ is disjoint $K_\sigma$, i.e.\ can be expressed as the union of countably many pairwise disjoint compact sets. It follows that $A$ can be expressed in the form $A=\bigcap_{n=1}^\infty B_n=\bigcap_{n=1}^\infty \bigcup_{m=1}^\infty F_n^m$ where each $F_n^m$ is compact and nowhere dense, and for any $n\in \N$ and any distinct $m_1,m_2\in\N$ we have $F_n^{m_1}\cap F_n^{m_2}=\emptyset$. 

We shall now recursively construct (for all $n\in\N$) the systems $(D_n^m)_{m=1}^\infty$ of pairwise disjoint compact sets with $\bigcup_{m=1}^\infty D_n^m=B_n$ such that 
\begin{equation}\label{E:OrderSetsConditionMainLemma}
  \text{for any $n\geq 2$ and any $m$ there is $k$ such that $D_{n}^m \subseteq D_{n-1}^k$.}    
\end{equation}
(If we consider the lower index $n$ as ``level'', then this requirement means that each set of level $n$ is contained in some set of level $n-1$.) We start with level $n=1$ where for each $m\in\N$ we set $D_1^m\eqdef F_1^m$.

Assuming we have already performed the construction up to level $n-1$ for some $n\geq 2$, we now take all the sets of the form $F_n^i\cap D_{n-1}^j$, $i,j\in\N$, discard the ones that are empty and order the remaining ones into a sequence $(D_n^m)_{m=1}^\infty$.

Having finished this construction for every $n$, we need to check that the above conditions on the sets $D_n^m$ are all met. The pairwise disjointness, compactness, and condition \cref{E:OrderSetsConditionMainLemma} are obvious; by induction we shall prove, for every $n\in\N$, the equality $\bigcup_{m=1}^\infty D_n^m = B_n$.

For $n=1$ the statement is obvious, so assume we have already proved $\bigcup_{m=1}^\infty D_{n-1}^m = B_{n-1}$ for some $n\geq 2$. Take any $x\in B_n$; then $x\in B_{n-1}$ by the nestedness of $B_n$'s, so there exists $j\in \N$ such that $x\in D_{n-1}^j$. Moreover, as $x\in B_n=\bigcup_{m=1}^\infty F_n^m$, there also exists $i$ such that $x\in F_n^i$. Thus $x\in F_n^i\cap D_{n-1}^j$, so this last intersection is one of the nonempty sets that get arranged into the sequence $(D_n^m)_{m=1}^\infty$ in the $n$-th step of the above recursive construction. Thus $x\in \bigcup_{m=1}^\infty D_n^m$, and we get $B_n\subseteq\bigcup_{m=1}^\infty D_n^m$; the opposite inclusion follows immediately from the fact $\bigcup_{m=1}^\infty D_n^m\subseteq \bigcup_{m=1}^\infty F_n^m=B_n$.

It is now easy to order all the sets $D_n^m$ ($n,m\in\N$) into a sequence $(F_n)_{n=1}^\infty$ satisfying conditions \cref{C:b},~and~\cref{C:c}, and we set $F_0=\R$, satisfying \cref{C:a}.


Since $A$ is Lebesgue null, it easily follows from \cref{L:SetM} that we can find pairwise disjoint measurable sets $M_k$ ($k\in\N$) such that for every $k$, we have $A\cap M_k=\emptyset$ and $\abs{I\cap M_k}>0$ for every $k$ and every interval~$I$. Now, for each $k\geqslant 1$, we use ~\cref{L:SetH} to obtain a compact set $H_k$ such that (cf.~\cref{N:FepsFhat}) 
\begin{itemize}
  \item $F_k\subseteq H_k \subseteq (F_k \cup M_k)\cap (F_k)_{2^{-k}}$;
  \item $H_k$ meets the middle third of every component of $\hat{F}_k^c\cap(F_k)_{2^{-k}}$ in a set of positive measure.
\end{itemize}
Next, we define $H_0=\R$, $\future{H}_0=\R$, $\past{E}_0=\emptyset$, and for $k\geqslant 1$ we set
\begin{equation}\label{E:defHkEk}
  \future{H}_k \eqdef \bigcup_{F_j\subseteq F_k} H_j \quad\text{ and } \quad \past{E}_k\eqdef \bigcup_{j<k, \; F_j\cap F_k=\emptyset} \future{H}_j.
\end{equation}
We now argue that the sets $\future{H}_k$ (and thus also the sets $\past{E}_k$) are all closed. 
Indeed, recall that given $k$, for each $j\geqslant k$ with $F_j\subseteq F_k$ we have $H_j \subseteq (F_j)_{2^{-j}}\subseteq (F_k)_{2^{-j}}$. Let $x\notin \future{H}_k$; then $\alpha\eqdef d(x,F_k)>0$ and we can pick $j_0>k$ such that $2^{-j_0}<\alpha/2$. It follows that for any $j\geqslant j_0$ with $F_j\subseteq F_k$, $d(x,H_j)\geqslant \alpha/2$. Thus
\begin{equation*}
  d(x,\future{H}_k)\geqslant\min\bigl(\{d(x,H_j)\setsep F_j\subseteq F_k \text{ and }j<j_0\}\cup\{\alpha/2\}\bigr)>0,
\end{equation*}
and we see that $\future{H}_k$ is closed.

Finally, given $k\geqslant 1$, we recall that we obtained in~\cref{L:functiong} a nondecreasing absolutely continuous function $g_k\colon \R\to [0,2^{-k}]$ satisfying the following properties:
\begin{enumerate}[(i)]
  \item\label{C:gk1} $\spt (g'_k)\subseteq H_k$;
  \item\label{C:gk2} $\Lip g_k(x)<\infty $ for every $x\in\R$;
  \item\label{C:gk3} $g'_k(x)=1$ for every $x\in A\cap F_k \cap \past{E}_k^c$;
  \item\label{C:gk4} $\osc(g_k, I)<2^{-k} \abs{I}$ whenever $I$ is an interval meeting $\past{E}_k$;
  \item\label{C:gk5} $\norm{g_k'}_{1}<2^{-k}$.
\end{enumerate}
\medskip

We will show that the statement holds with $g=\sum_{k=1}^{\infty} g_k$.
That $g$ is nondecreasing is clear, and the absolute continuity is a consequence of~\cref{sum AC}.

First, we prove that $g'(x)=\infty$ for every $x\in A$. 
We do this by using~\cref{C:gk3}.
As $x\in A$, there are infinitely many indices~$k$ such that $x\in F_k$.
We fix such an index~$k$, and hence have $x\in A\cap F_k$.
In view of~\cref{C:gk3}, we want to show that $x\notin \past{E}_k= \bigcup_{j<k, \; F_j\cap F_k=\emptyset} \future{H}_j$.
To that end, suppose that $1\leqslant j<k$  and $F_j\cap F_k =\emptyset$. Since $A\cap \bigcup_{i=1}^\infty M_i=\emptyset$, we have
\begin{equation*}
  x\notin \bigcup_{i=1}^\infty M_i\supseteq \bigcup_{i=1}^\infty (H_i\setminus F_i) \supseteq \bigcup_{F_i\subseteq F_j} (H_i\setminus F_i) \supseteq \bigcup_{F_i\subseteq F_j} (H_i\setminus F_j) = \future{H}_j\setminus F_j,
\end{equation*}
so $x\notin \future{H}_j\setminus F_j$, which together with $x\in F_k$ and $F_j\cap F_k=\emptyset$ implies $x\notin \future{H}_j$, which yields $x\notin \past{E}_k$.

Overall, for all $k\in\N$ such that $x\in A\cap F_k$ we have, in fact, that $x\in A\cap F_k \cap \past{E}_k^c$, so $g_k'(x)=1$. Since, by~\cref{C:c}, there are infinitely many such indices $k$, and all the functions $g_k$ are nondecreasing, we easily obtain that $g'(x)=\infty$.
\medskip

Having finished the proof in the first case, let us now assume $x\notin A$; we aim to prove $\lip g(x)<\infty$. By~\cref{C:c} there are only finitely many indices~$k$ with $x\in F_k$. Moreover, since the sets $M_k$ are pairwise disjoint, there is at most one $k$ with $x\in M_k$. It immediately follows that there are finitely many $k$ such that $x\in H_k$ (as $H_k\subseteq F_k\cup M_k$); thus we can define $l$ to be the largest index with $x\in H_l$. Now we give an argument showing that, in fact, $l$ is also the largest index~$k$ with $x\in \future{H}_k$. Indeed, if we had $p>l$ with $x\in \future{H}_p = \bigcup_{F_j\subseteq F_p} H_j$, then there would exist $j$ such that $F_j\subseteq F_p$ (whence $j\geq p$ by~\cref{C:b}) and $x\in H_j$. Since $j\geq p>l$, this would be in contradiction with our choice of $l$.
\medskip

To prove $\lip g(x)<\infty$, we divide $g=\sum_{j=1}^{\infty} g_j$ into three summands as we now describe.
First, we introduce
\begin{align*}
    \JJ &= \{j>l\setsep x\notin \past{E}_j\},\\
    \KK &= \{j>l\setsep x\in \past{E}_j\}
\end{align*}
and then write
\begin{equation*}
    g=\sum_{j\in \JJ} g_j+\sum_{j\in \KK} g_j+\sum_{j=1}^l g_j.
\end{equation*}
We will show that the first summand has finite $\lip$ at~$x$, while the other two even have finite $\Lip$ at this point.
We define
\begin{equation*}
  \quad h=\sum_{j\in\JJ} g_j,
\end{equation*}
and want to show that $\lip h(x)<\infty$. If there is $r>0$ such that $(x-r,x+r)\cap \bigcup_{j\in\JJ}\future{H}_j = \emptyset$, then \cref{C:gk1} implies $h$ to be constant on $(x-r,x+r)$, so $\lip h(x)=0$, and we are done. 

Hence we assume
\begin{equation}\label{E:NontrivInts}
  (x-r,x+r)\cap \bigcup_{j\in\JJ}\future{H}_j \neq \emptyset\quad\text{ for every }r>0.
\end{equation}

It suffices to find a sequence~$(r_p)_{p=1}^\infty$ of positive radii converging to~$0$, and for which
\begin{equation*}
  \lim_{p\to \infty} \sup_{y\in B(x,r_p)}\frac{\abs{h(y)-h(x)}}{r_p}<\infty.
\end{equation*}
Let $r_0=1$, and for $p=1,2,\dots$, 
we define recursively $j_p\in\JJ$ and $r_p>0$ by letting 
\begin{align*}
j_p &=\min\{j\in\JJ\setsep (x-r_{p-1},x+r_{p-1})\cap \future{H}_j\neq\emptyset\},    \\
r_p &=\dist(x,\future{H}_{j_p}),\\
I_p &=(x-r_p,x+r_p).
\end{align*}

By~\cref{E:NontrivInts}, $j_p$~and~$r_p$ are well-defined, and it is easy to see that $j_1<j_2<\dotsb$,
and $r_0>r_1>\dotsb$.
We now argue by contradiction that $R\eqdef\lim_{p\to\infty} r_p=0$; so assume not, i.e.\ $R>0$. Take $k\in \JJ$ such that $(x-R,x+R)\cap \future{H}_k\not=\emptyset$. Since $\lim_{p\to \infty}j_p=\infty$,
we may pick $p\in \field{N}$ with $k<j_p$. As $r_{p-1}>R$, we have $(x-r_{p-1},x+r_{p-1})\cap \future{H}_k\neq \emptyset$; therefore, the fact that $k<j_p$ is in contradiction with the definition of $j_p$.

Next we note that $F_{j_p}\cap F_{j_q}=\emptyset$ whenever $p\neq q$: Assume, for a contradiction, $p<q$ and $F_{j_p}\cap F_{j_q}\neq\emptyset$. Then $j_p<j_q$, and \cref{C:b} yields $F_{j_p}\supseteq F_{j_q}$, whence $\future{H}_{j_p}\supseteq \future{H}_{j_q}$, implying $r_p\leqslant r_q$, a contradiction.

Fix arbitrary $p\in \N$ and $j\in\JJ$; we shall estimate the oscillation of $g_j$ on $I_p$. If $I_p\cap H_j=\emptyset$, then $g_j$ is constant on $I_p$; hence, we shall assume $I_p\cap H_j \neq \emptyset$.   
Then $I_p\cap\future{H}_j\neq\emptyset$, and $r_p>\dist(x,\future{H}_j)$. By the construction of $j_p$ and $r_p$, this yields $j_p<j$ and $j_{p+1}\leqslant j$. 

We verify next that $F_j\cap F_{j_p}=\emptyset$. Indeed, by~\cref{C:b} and as $j>j_p$, the only alternative is $F_j\subseteq F_{j_p}$; but in that case we would also have $\future{H}_j\subseteq \future{H}_{j_p}$, which is impossible as $\future{H}_j$ meets $I_p$ and $\future{H}_{j_p}$ (by the definition of $I_p$) does not.


To summarize, we have fixed arbitrary $p\in\N$, $j\in\JJ$, and assumed $I_p\cap H_j\neq\emptyset$, obtaining $j_p<j_{p+1}\leqslant j$ and $F_j\cap F_{j_p}=\emptyset$. Thus
\begin{equation*}
  \past{E}_j = \bigcup_{k<j, F_k\cap F_j =\emptyset} \future{H}_k \supseteq \future{H}_{j_p},
\end{equation*}
and since clearly the closure $\overline{I_p}$ meets $\future{H}_{j_p}$, it therefore also meets $\past{E}_j$. Property~\cref{C:gk4} of $g_j$ yields the estimate
\begin{equation*}
\osc(g_j, I_p)=\osc\bigl(g_j, \overline{I_p}\bigr)\leqslant 2^{-j}\cdot \abs{\overline{I_p}} =2^{-j}\cdot \abs{I_p}.
\end{equation*}
Since this estimate holds for any $j\in\JJ$, by summing over $j\in\JJ$, and using (in the first equality) that all $g_j$'s are nondecreasing, we obtain
\begin{equation*}
  \osc(h,I_p) = \sum_{j\in\JJ} \osc(g_j, I_p) \leqslant \sum_{j\in\JJ} 2^{-j} \abs{I_p} \leqslant \abs{I_p}.
\end{equation*}
This estimate holds for every $p\in\N$, and since $r_p\to 0$ (i.e.\ $\abs{I_p}\to 0$) as $p\to\infty$, we conclude
\begin{equation*}
  \begin{split}
  \lip h(x) = \liminf_{R\to 0_+} \;\sup_{y\in B(x,R)} \frac{\abs{h(y)-h(x)}}{R}\leqslant \liminf_{p\to\infty} \;\sup_{y\in I_p} \frac{\abs{h(y)-h(x)}}{\abs{I_p}/2} \\
  \leqslant 2 \liminf_{p\to\infty} \frac{\osc(h,I_p)}{\abs{I_p}} \leqslant 2.
  \end{split}
\end{equation*}
\medskip

Now, take any $j>l$ such that $x\in \past{E}_j$ (i.e.\ $j\notin \JJ$); then any interval containing $x$ meets $\past{E}_j$. Thus, for any $r>0$, $\osc\bigl(g_j,(x-r,x+r)\bigr)< 2^{-j}\cdot 2r= 2^{1-j}\cdot r$. Setting $f=\sum_{j>l,\, j\notin \JJ} g_j$, we obtain
\begin{align*}
  \Lip f(x) &=\limsup_{y\to x}\frac{\abs{f(y)-f(x)}}{\abs{y-x}} \leqslant \limsup_{r\to 0_{+}}\frac{\osc\bigl(f,(x-r,x+r)\bigr)}{r}\\
  &= \limsup_{r\to 0_{+}}\frac{\sum_{j>l,\,j\notin \JJ}\osc\bigl (g_j,(x-r,x+r)\bigr)}{r}\\
  &\leqslant \limsup_{r\to 0_{+}}\frac{2^{1-j}r}{r} \leqslant \sum_{j=2}^{\infty} 2^{1-j}=1.
\end{align*}

Finally, as $g=h+f+\sum_{j=1}^{l} g_j$, we obtain the following estimate:
\begin{equation*}
  \lip g(x)\leqslant \lip h(x)+ \Lip f(x)+\sum_{j=1}^{l}\Lip g_j(x) \leq 2+1+\sum_{j=1}^{l}\Lip g_j(x)<\infty.
\end{equation*}

This concludes the proof.
\end{proof}

Finally, we are ready for the proof of our main result.
\begin{proof}[Proof of~\cref{T:main}]
  Since all Borel sets have the Baire property (see e.g.\ \cite[Theorem~4.3]{Oxtoby}), by \cite[Theorem~4.4]{Oxtoby}, there are a meagre $F_{\sigma\delta}$ set $A_0$ and a $G_\delta$ set $G$ such that $A_0\cap G=\emptyset$ and $A_0\cup G= A$ (the disjointness is not clearly stated in \cite{Oxtoby}, but can easily be seen from the proof).
  
  Jarník constructed a nondecreasing absolutely continuous function $f\colon \field{R}\to \field{R}$ such that $f'(x)=\infty$ for every $x\in G$ and $\Lip f(x)<\infty$ for every $x\notin G$.
  Note that this also shows that $\Lip f(x)=\infty$ for $x\in G$.
  From the meagre case,~\cref{null}, we know that there is a nondecreasing absolutely continuous function $g\colon \field{R}\to \field{R}$ such that $g'(x)=\infty$ for every $x\in A_0$ and $\lip g(x)<\infty$ for $x\notin A_0$.
  Again, we note that $\Lip g(x)=\infty$ for every $x\in A_0$.

  We define $h\eqdef f+g$.
  It is clear that $h$ is nondecreasing and absolutely continuous.
  If $f'(x)=\infty$ or $g'(x)=\infty$, then also $h'(x)=\infty$.
  This shows that $h'(x)=\infty$ for $x\in A$, and hence that $\lip h(x)=\infty$.

  Now, we assume that $x\notin A$.
  Hence, $\Lip f(x)<\infty$ and $\lip g(x)<\infty$.
  This tells us that $\lip h(x)<\infty$ for $x\notin A$.

  We have found a function with all the required properties.
\end{proof}

\textbf{Acknowledgments}
This project started while we were at the University of Warwick as members of a research group led by David Preiss.
We want to give our heartfelt thanks for his mentoring and for the many discussions about numerous aspects of life and mathematics, and in particular about questions concerning~$\lip$.
We are grateful to him for suggesting the idea of our proof.

The research leading to these results has received
funding from the European Research Council under the European Union’s
Seventh Framework Programme (FP/2007–2013) / ERC Grant Agreement
no.~291497. Furthermore, the second author was supported
by the University of Silesia Mathematics Department (Iterative Functional
Equations and Real Analysis program).

\bibliographystyle{amsalpha}
\bibliography{lip}
\end{document}